\documentclass[12pt]{extarticle}
\usepackage[utf8]{inputenc}
\usepackage[T1]{fontenc}

\usepackage{caption}
\oddsidemargin \evensidemargin
\usepackage{amssymb}
\usepackage{amsmath}
\usepackage{comment}
\usepackage{amsthm}
\usepackage{amssymb}
\usepackage{algorithmicx}
\usepackage{algorithm}
\usepackage{algpseudocode}
\usepackage{mathtools}
\usepackage[margin=1cm]{caption}
\usepackage{subcaption}
\usepackage{mathrsfs}
\usepackage{multirow}
\usepackage[dvipsnames]{xcolor}
\usepackage{tikz,tikz-3dplot,tikz-cd}
\usepackage{pgfplots}
\usepackage{url}
\usepackage{csquotes}
\usepackage{MnSymbol}
\usepackage[backend=biber,
style=alphabetic,doi=false]{biblatex} 
\usepackage[normalem]{ulem}

\usepackage{hyperref}
\hypersetup{
    colorlinks=true,
    linkcolor=orange!80!black,
    filecolor=magenta,      
    urlcolor=blue,
    citecolor=blue,
}

\definecolor{mycolor1}{rgb}{0.00000,0.44700,0.74100}
\definecolor{mycolor2}{rgb}{0.8500, 0.3250, 0.0980}
\definecolor{mycolor3}{rgb}{0.9290, 0.6940, 0.1250}
\definecolor{mycolor4}{rgb}{0.4940, 0.1840, 0.5560}
\definecolor{mycolor5}{rgb}{0.4660, 0.6740, 0.1880}

\newtheorem{definition}{Definition}[section]

\newtheorem{proposition}[definition]{Proposition}
\newtheorem{corollary}[definition]{Corollary}
\newtheorem{theorem}[definition]{Theorem}

\newtheorem{lemma}[definition]{Lemma}

\newtheorem*{theorem*}{Theorem}

\theoremstyle{remark}
\newtheorem{remark}[definition]{Remark}
\newtheorem{examplex}[definition]{Example}
\newenvironment{example}
  {\pushQED{\qed}\examplex}
  {\popQED\endexamplex}

\usepackage{mathtools} 
\mathtoolsset{showonlyrefs,showmanualtags}
\numberwithin{equation}{section}

\newcommand{\R}{\mathbb{R}}

\newcommand{\C}{\mathbb{C}}
\newcommand{\N}{\mathbb{N}}

\newcommand{\PP}{\mathbb{P}}

\newcommand{\cV}{\mathcal{V}}

\newcommand{\Hom}{\operatorname{Hom}}

\newcommand{\wt}{\widetilde}

\DeclareMathOperator{\gr}{Gr}
\DeclareMathOperator{\ch}{CH}
\DeclareMathOperator{\Div}{Div}

\newcommand{\aA}{\mathcal{A}}
\newcommand{\CH}{\operatorname{CH}}
\newcommand{\Sing}{\operatorname{Sing}}
\newcommand{\Sec}{\mathcal{S}_{01}}
\newcommand{\secV}{\operatorname{Sec}}

\newcommand{\open}{\operatorname{int}}

\newcommand{\sO}{\mathcal{O}}
\newcommand{\grd}{\operatorname{Gr}(d{-}1,d{+}1)}

\title{Taking the amplituhedron to the limit}
\author{Joris Koefler, Rainer Sinn}
\date{December 2024}

\begin{document}

\maketitle
\begin{abstract}
    The amplituhedron is a semialgebraic set given as the image of the non-negative Grassmannian under a linear map subject to a choice of additional parameters. 
    We define the limit amplituhedron as the limit of amplituhedra by sending one of the parameters, namely the number of particles $n$, to infinity. 
    We study this limit amplituhedron for $m=2$ and any $k$, relating to the number of negative helicity particles.
    We determine its algebraic boundary in terms of Chow hypersurfaces. These hypersurfaces in the Grassmannian are stratified by singularities in terms of higher order secants of the rational normal curve.
    In conclusion, we show that the limit amplituhedron is a positive geometry with a residual arrangement that is empty.
\end{abstract}

\section{Introduction}\label{sec:intro}

Amplituhedra were introduced by Arkani-Hamed and Trnka \cite{Amplituhedron} in the context of scattering amplitudes in particle physics. We can view them as a generalization of polytopes to a nonlinear setting as follows. Every polyhedral cone is the image of a non-negative orthant under a linear map (sometimes called $\mathcal{V}$-representation). 
Analogously, a (tree-level) amplituhedron is defined as the image of a non-negative Grassmannian under a linear map; the non-negative orthant $\R^d_{\geq 0}$ is the special case of the non-negative Grassmannnian in $\gr(1,d) \cong \PP^{d-1}$. 
One major open question in this area asks if amplituhedra are positive geometries in the sense of \cite{canForms_Lam_Hamed}. The answer is known to be affirmative for polytopes, see \cite{lam2022invitationpositivegeometries} for more context. 
In \cite{ranestad2024adjoints}, the authors show that the amplituhedron is a positive geometry in the special case of $k=m=2$, which is a toy model for the $m=4$ amplituhedron which encodes the scattering amplitudes of $n$ particles, $k+2$ of them with negative helicity, in $\mathcal{N}=4$ SYM.
In our work, we consider any $k$ and take that result to the limit, as the number of particles $n$ goes to infinity. 
We define the limit amplituhedron, which is a semialgebraic subset of the Grassmannian $\gr(k,k{+}2)$, and show that it is a positive geometry. 
A simple case of this construction is the mysterious \enquote{pizza slice}, see Example~\ref{ex:pizza_slice}.
It is also special version of the master amplituhedron that already appeared in \cite{Amplituhedron}.

The definition of a positive geometry is recursive in dimension, see \cite{lam2022invitationpositivegeometries}. In the case of polytopes, this means that we successively take residues of the canonical differential form along the flats of the facet hyperplane arrangement. 
In the end, we need to arrive at the form $\pm 1$ at the vertices of the polytope and to show that there are no other poles.
In the nonlinear setting here, we need to replace the facet hyperplane arrangement of a polytope by the algebraic boundary of the limit amplituhedron.
Its algebraic boundary is a hypersurface in the Grassmannian and the flats in the linear case correspond to iterated singular loci. The algebraic boundary of the limit amplituhedron is the union of two Chow hypersurfaces in the Grassmannian, namely the Chow hypersurface of the rational normal curve and a secant line.
We describe the iterated singular loci geometrically in terms of varieties of secant planes and osculating planes to the rational normal curve. 

Our paper is structured along these lines. In Section~\ref{sec:algebraic_bd} we introduce the limit amplituhedron formally, establish some basic properties, and determine its algebraic boundary in terms of the Chow hypersurfaces of the rational normal curve and a line in $\PP^{d}$. In Section~\ref{sec:stratification} we describe all strata of the iterated singular loci of the algebraic boundary, generalizing the matroid of the facet hyperplane arrangement of a polytope to our setting. Finally, in Section~\ref{sec:residual_arrangement}, we show that the residual arrangement of the limit amplituhedron is empty. 
In particular, we prove that all $0$-dimensional strata are vertices of the limit amplituhedron, such that the canonical form given in terms of the adjoint hypersurface of the algebraic boundary satisfies the requirements for a positive geometry.

In summary, our main result states the following.
\begin{theorem}\label{thm:main_pos_geom}
    The limit amplituhedron $\aA_k^\infty$ is a positive geometry  $(\gr(k,k{+}2),\aA^\infty_k)$ in the Grassmannian $\gr(k,k+2)$. Its algebraic boundary has two irreducible components, the Chow hypersurface of the rational normal curve $C_{k+1}\subset \PP^{k+1}$ and the Chow hypersurface of the secant line $\Sec\subset \PP^{k+1}$ spanned by two points on $C_{k+1}$.
\end{theorem}

The special case $k=1$ gives a positive geometry in $\PP^2$ that also appears in \cite[{Example 6.1}]{canForms_Lam_Hamed} under the name of pizza slice.
\begin{example}[Pizza slice]\label{ex:pizza_slice}
For $k=1$, we pick a totally positive matrix $Z \in \R^{n\times 3}$ (that is a real $n\times 3$ matrix such that all its $3\times 3$ minors are positive). In fact, we pick the columns of $Z$ to be on the rational normal curve in $\PP^2$ by choosing real numbers $0\leq t_1<t_2<\ldots<t_n\leq 1$ and defining the $i$-th column of $Z$ to be $\gamma_2(t_i) = (1,t_i,t_i^2)$.
The amplituhedron $\aA_1^{t_1,\ldots,t_n}$ is, by definition, the image $\wt{Z}(\gr(1,n)_{\geq 0})$ of the non-negative Grassmannian $\gr(1,n)_{\geq 0}$ under the linear map $\wt{Z}\colon \gr(1,n)_{\geq 0} \to \R^3$, $V\mapsto VZ$.
Since $\gr(1,n)_{\geq 0}$ is a non-negative orthant, the image is a $n$-gon $\PP^2$ with vertices given by the rows of the matrix $Z$.
So the vertices are all on the rational normal curve in the projective plane that we see as a parabola on the affine chart that we chose by parametrizing it via $t\mapsto (1,t,t^2)$. To define the limit amplituhedron $\aA_1^\infty$, we take the union over all amplituhedra  $\aA_1^{t_1,\ldots,t_n}$ that we obtain by choosing any $n$ real numbers in $[0,1]$ and any $n\in \N$. This union is the convex hull of the interval $\{ (1,t,t^2) \mid t\in [0,1]\}$ on the rational normal curve. It is projectively equivalent to the set shown in Figure~\ref{fig:Bd_2_infinituhedron}, essentially equal to the \enquote{pizza slice} in \cite[Example 6.1]{canForms_Lam_Hamed}.
The algebraic boundary of $\aA_1^\infty$ has two irreducible components, namely the rational normal curve itself and the line spanned by $\gamma_2(0)$ and $\gamma_2(1)$. For higher values of $k$, the two components are Chow hypersurfaces; for $k=1$, in $\PP^2$ they coincide with the varieties themselves. 
    \begin{figure}[]
        \centering
        \begin{tikzpicture}
    \path[] (0,0) circle(2cm);
    \path[] (-3,0) -- (3,1);

    \coordinate (A) at (-1.992,0.167)  {};
    \coordinate (B) at (1.83,0.805) {};

    \begin{scope}[rotate=14]
        \fill[Green!20] (A) arc[start angle=160, end angle=371, radius=2cm] -- cycle;
    \end{scope}

    \draw[very thick] (0,0) circle(2cm);
    \draw[very thick] (-3,0) -- (3,1) ;
    \node[] at (-3,0.3) {$\Sec$};
    \node[] at (1.5,1.9) {$C_2$};
    \node[] at (0,-1) {$\aA^\infty_1$};

\end{tikzpicture}
        \caption{The limit amplituhedron $\aA^\infty_1$ with its algebraic boundary.}
        \label{fig:Bd_2_infinituhedron}
    \end{figure}
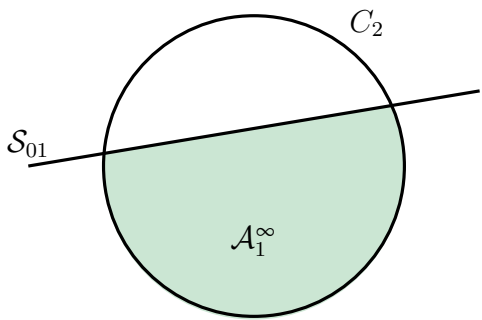
\end{example}

This example already illustrates that we need to choose an interval $[0,1]$ on the rational normal curve. If we dropped our restriction of $t_i$ being in $[0,1]$, the limit amplituhedron would be the convex hull of the rational normal curve, so up to a projective change of coordinates the unit disk. This is, however, not a positive geometry in $\PP^2$ because it cannot have a rational canonical form (for instance for degree reasons: there is no rational section of the canonical divisor of $\PP^2$ with a denominator of degree $2$), see \cite[Section 10]{canForms_Lam_Hamed}. However, adding the secant line $\Sec$ increases the degree by $1$ and we get a rational canonical form as in the definition of a positive geometry. The same happens in Grassmannians for higher values of $k$: we need to pick an interval on the rational normal curve so that we get a special secant line between the end points that gives the correct degree of the denominator to have a chance of finding a rational canonical form. The first nonlinear example is for $k=2$ which gives a limit amplituhedron in $\gr(2,4)$ defined in terms of the twisted cubic $C_3\subset \PP^3$, see Example \ref{ex:twisted_cubic}.\\

\noindent\textbf{Acknowledgements.} We would like to thank Daniele Agostini, Thomas Lam and Kristian Ranestad for helpful pointers.

\section{Algebraic boundary}\label{sec:algebraic_bd}

The \emph{non-negative Grassmannian} $\gr(k,n)_{\geq 0}$ is the semialgebraic subset of $\gr(k,n)$ defined as the intersection of $\gr(k,n)$ in its Pl\"ucker embedding to $\PP^{ \binom{n}{k}-1}$ with the non-negative orthant; equivalently, it consists of all real matrices $V\in \R^{k\times n}$ of rank $k$ whose $k\times k$ minors are all non-negative. 
A (tree-level) amplituhedron depends on a choice of parameters $k,m$, and $n$ and the choice of a \emph{totally positive} $n\times (k{+}m)$ matrix $Z$, that is all its maximal minors are positive. In this paper we are going to fix $m=2$. Then, such a matrix $Z$ induces a rational map
\[
\widetilde{Z}:\gr(k,n) \dashedrightarrow \gr(k,k{+}2)\quad V\mapsto VZ.
\]
By the positivity assumptions, this map is well-defined on $\gr(k,n)_{\geq 0}$. The image $\wt{Z}(\gr(k,n)_{\geq 0}$ of the non-negative Grassmannian under this map is the \emph{tree-level amplituhedron} $\aA_{n,k,2}(Z)$, first defined by Arkani-Hamed and Trnka \cite{Amplituhedron}. 
We construct the totally positive matrices $Z$ via the rational normal curve in $\PP^{k+1}$.
Let $\gamma_{k+1}:\mathbb{R}\rightarrow \PP^{{k+1}}\cap U_{0}$ sending $t\mapsto(1,t,t^2,\dots, t^{k+1})$, be the parametrisation of the rational normal curve $C_{k+1}$ of degree ${k+1}$ on the affine patch $U_0$ corresponding to a non-vanishing first coordinate in $\mathbb{P}^{k+1}$.
A \emph{partition} $I_n$ of the interval $[0,1]\subset\R$ of length $n$ is an increasingly ordered sequence of real numbers $0=t_1<\ldots<t_n=1$.
For fixed positive integers $k$ and $n$, and any partition $I_n = (t_1,t_2,\ldots,t_n)$ of $[0,1]$, we define the real totally positive $n\times (k{+}2)$ matrix 
\[
Z = Z(I)= \begin{pmatrix}
    \gamma_{k+1}(t_1) \\
    \vdots \\
    \gamma_{k+1}(t_n)
\end{pmatrix},
\]
whose $i$-th row is given by $\gamma_{k+1}(t_i)$. The maximal minors of this matrix are Vandermonde matrices and therefore $Z(I)$ is a totally positive matrix. 
We denote the amplituhedron which is the image of the induced map $\widetilde{Z}$ by $\aA^I_k$.
This notation coincides with the previous one as $\aA^I_k=\aA_{|I|,k,2}(Z(I))$.

Next, we define the limit amplituhedron. A refinement $J_{m}$ of a partition $I_n$ of $[0,1]$ is another ordered sequence of $m$ real numbers which contains the ordered sequence of $I_n$ as a sub-sequence.
The \emph{limit amplituhedron}, $\aA_k^\infty$ is the union of $\aA^{I}_k$ over all partitions $I$ of the interval $[0,1]$.

\begin{remark}
    The totally positive matrices $Z(I)$ that we consider here are special. They correspond to realizations of cyclic polytopes in $\R^{k+1}$ with $n$ vertices on the rational normal curve. The convex hull of $n$ points on the rational normal curve is always a cyclic polytope but not every cyclic polytope is affinely equivalent to such a realization (as soon as the number of vertices is big enough), see \cite{cyclic_poly} for more details. We restrict to totally positive matrices $Z(I)$ constructed from the rational normal curve to get a better handle on the limit amplituhedron. 
\end{remark}

\begin{proposition}
    Let $\mathcal{I}$ be the set of all partitions of $[0,1]$ and $\leq$ the refining order on $\mathcal{I}$, that is order induced by inclusions. Then, there is a well defined direct limit with
    \begin{align}
        \lim\limits_{\longrightarrow}
        \aA^I_k = \aA^\infty_k,
    \end{align}
    where $I\in\mathcal{I}$.
\end{proposition}
\begin{proof}
    Let $\mathcal{I}$ and $\leq$ be as above. First, notice that $\leq$, as the order induced by containment of subsequences, is reflexive and transitive, thus a preorder, implying that $(\mathcal{I},\leq)$ is a directed set.
    Next, for partitions $I\leq J\in\cal I$ let $D=J\setminus I$ be their refining set.
    Then, for each $V_I=(v_1,\dots, v_{|I|})\in\gr(k,|I|)_{\geq 0}$ we can define $V_J = (v'_1,\dots, v'_{|J|})\in\gr(k,|J|)_{\geq 0}$ where $v'_i = v_i$ if $i\notin D$ else $v'_i = 0$.
    Thus, $V_IZ(I) = V_JZ(J)$, and we get well defined inclusion maps
    \begin{align*}
        \iota_{IJ}: \aA^I_k\hookrightarrow \aA^J_k.
    \end{align*}
    Therefore, $\langle \aA^I_k,\iota_{IJ}\rangle$ is a directed system, as it also satisfies the additional axioms $\iota_{II}=\operatorname{id}$ and $\iota_{IK} = \iota_{JK}\circ\iota_{IJ}$ for all $I\leq J \leq K \in \mathcal{I}$; hence the direct limit is well defined and the inclusion in $\aA^\infty_k$ is by definition.
\end{proof}  

Next we briefly want to examine the properties of the Euclidean boundary $\partial\aA^\infty_k$, where we write $\open(-)$ for the Euclidean interior operator on $\gr_\R(k,k{+}2)$.

\begin{proposition}\label{prop:int_A_infty}
Let $\mathcal{I}$ be the set of all partitions of the interval $[0,1]\subset\R$.
Then, the Euclidean interior $\open(\aA^{\infty}_k)$ of the limit amplituhedron is given by the union of the Euclidean interiors $\open( A^{I}_k)$ over all $I\in\mathcal{I}$ i.e.
\[
\open\left(\aA^{\infty}_k\right) = \bigcup_{I\in\mathcal{I}}\open (A^{I}_k).
\] 
\end{proposition}
\begin{proof}
    We show both inclusions.
    First, notice that any interior point of $\aA^I_k$ for a fixed partition $I$ is clearly an interior point of $\aA^\infty_k$.
    For the other inclusion, let $p\in\open\aA^\infty_k$ be an interior point.
    Thus, there exists an open set $U\subset \PP^{N-1}$, with $N=\binom{k{+}2}{k}$, such that $p\in U\subsetneq \aA^\infty_k$, moreover we can choose $U$ even smaller, such that there is another open set $V\subsetneq \aA^\infty_k$ with $\operatorname{cl}(U)\subset V$, where $\operatorname{cl}(\cdot)$ denotes the Euclidean closure operator on $\PP^{N-1}$.
    Then, $\bigcup_{I\in\mathcal{I}}\open(V\cap\aA^I)\supset U$ is an open cover of $U$, as $\open (V)\supset U$ following from the monotonicity of the Euclidean interior operator.
    Moreover, since $\aA^\infty_k$ is compact by Tychonoff's theorem, there exists a finite subcover labelled by $\mathcal{F}\subset\mathcal{I}$ with $U\subset\bigcup_{I\in\mathcal{F}}\open(V\cap \aA^I_k)$.
    If there exists a $I_0\in\mathcal{F}$ such that $U\subset \open(V\cap\aA^{I_0}_k)$ then we are done because it is an open set of $\PP^{N-1}$ with $p\in \open(V\cap\aA^{I_0}_k)\subset \aA^{I_0}_k$, that is to say $p$ is an interior point of  $\aA^{I_0}_k$. 
    If there does not exists such a single $I_0$ we can collect all of the $I\in\mathcal{F}$ with $U\cap \open(V\cap\aA^I_k)\neq\varnothing$ in a finite subset $\mathcal{E}$ of $\mathcal{F}$. Then, we can denote by $\overline{I}$ the common refinement of all partitions in $\mathcal{E}$, and since taking interiors is monotone we then have $U\subset\open(V\cap\aA_{\overline{I}})$, what we ought to show.
\end{proof}

We write $\partial -$ for the Euclidean {boundary} operator on $\gr_\R(k,k{+}2)$. 
\begin{proposition}\label{prop:eucl_bd}
    The limit amplituhedron $\aA^\infty_k$ is a closed set in the Euclidean topology.
    Every plane in the Euclidean boundary $\partial \aA^\infty_k\subset \gr(k,k{+}2)$ intersects the rational normal curve $C_{k+1}$ of degree $k{+}1$ in the interval $\gamma_{k+1}([0,1])$ or it intersects the convex hull of $\{\gamma_{k+1}(0),\gamma_{k+1}(1)\}$ taken in the affine chart $U_0$ of $\PP^{k+1}$, which is an interval on the secant line $\Sec$. Conversely, every such plane is in the Euclidean boundary $\partial\aA^\infty_{k}$.
\end{proposition}
\begin{proof}
    By Proposition \ref{prop:int_A_infty} it follows that a point is in the boundary of $\aA^\infty_k$ if and only if it is a boundary point of $\aA^I_k$ for all partitions $I$ of $[0,1]$.
    On the other hand, it follows from Lemma \ref{lem:geom_alg_bd} that the (Euclidean) boundary strata $\partial\aA^I_k$ of $\aA^I_k$ are given by the intersection $\CH(\gamma_{k+1}(t_{i}),\ldots, \gamma_{k+1}(t_{i+k}))\cap \aA^I_k\subset\gr(k,k{+}2)$ for any $i\in[n]$ read cyclically.
    Therefore, a $k$-space is a boundary point of $\aA^\infty_k$ if and only if it intersects $\gamma_{k+1}([0,1])$ or the line segment $[\gamma_{k+1}(0),\gamma_{k+1}(1)]\subset\Sec$. Since these $k$-spaces are, by Proposition \ref{lem:limit_of_boundary_strata}, contained in $\aA^\infty_k$, the limit amplituhedron is a closed set with boundary points as described above.
\end{proof}

By $\overline{\:\cdot\:}$ we denote the Zariski closure of a subset in projective space. The \emph{algebraic boundary} $\partial_a\aA_{n,k}$ is the Zariski closure $\overline{\partial \aA_{n,k}}$ of the boundary of $\aA_{n,k}$ in the Euclidean topology of $\gr_\R(k,k{+}2)$.  We extend these definitions to the limit amplituhedron $\aA^\infty_k$.
The goal of this section is then to prove the decomposition of the algebraic boundary $\partial_a\aA^\infty_k$ of the limit amplituhedron $\aA^\infty_k$ in terms of the Chow hypersurfaces of the rational normal curve of degree $k{+}1$ and its fixed secant line $\Sec$ spanned by $\gamma_{k+1}(0)$ and $\gamma_{k+1}(1)$ in $\PP^{k+1}$.

\begin{definition}
    Let $L$ be a linear subspace in $\PP^{n}$ of dimension $n{-}d{-}1$ and $X$ an irreducible projective variety of dimension $d$. Then, we define the \emph{Chow hypersurface} $\CH(X)$ of $X$ as 
    \begin{align*}
        \CH(X) = \overline{\{L\mid X\cap L\neq \varnothing\}}\subset\gr(n{-}d,n{+}1).
    \end{align*}
\end{definition}
\begin{remark}
    As the name suggest the Chow hypersurface is defined by the vanishing of a unique polynomial, up to scaling, in the coordinate ring $\mathbb{C}[\gr(n{-}d,n{+}1)]$, referred to as the \emph{Chow form} $\operatorname{CF}(X)$ of $X$.
\end{remark}

The main result of this section, which is one part of Theorem \ref{thm:main_pos_geom}, is as follows:

\begin{theorem}\label{thm:algebraic_boundary_of_A_infty}
  The algebraic boundary $\partial_a\aA^\infty_k\subset \gr(k,k{+}2)$ of $\aA^\infty_k$ decomposes into the Chow hypersurface $\operatorname{CH}(C_{k+1})$ of the rational normal curve $C_{k+1}$ of degree $k{+}1$ and the Chow hypersurface $\CH(\Sec)$ of the secant line $\Sec$.
  Put algebraically,
  \[
  \partial_a\aA^\infty_k=\operatorname{CH}(C_{k+1})\cup \operatorname{CH}(\Sec)\subset \gr(k,k{+}2).
  \]
\end{theorem}
\begin{remark}\label{rem:physics_lingo}
    Firstly, notice that this implies that the limit amplituhedron $\aA^\infty_k$ is indeed a semialgebraic set. 
    Secondly, in order to match the notation more commonly found in Physics literature, notice that we can write
    \begin{align}
        \operatorname{CH}(\Sec)= \{(A_1\ldots A_k)\in\gr(k,k{+}2)\mid \langle A_1\ldots A_k\Sec\rangle=0\},
    \end{align}
    where $\langle A_1\ldots A_k\Sec\rangle$ denotes the determinant of the $(k{+}2)\times (k{+}2)$ matrix whose rows are given by the $A_i's$ and $\Sec$. In an abuse of notation we will also write $\CH(\Sec)=\langle A_1\ldots A_k\Sec\rangle$.
\end{remark}

Let us start by considering an explicit example.
\begin{example}\label{ex:twisted_cubic}
    For $k=2$ the limit amplituhedron $\aA^\infty_2$ lives in $\gr(2,4)$, and its algebraic boundary is given by the Chow hypersurfaces of the twisted cubic $C_{3}$ in $\PP^3$ and its secant line $\Sec$ spanned by $\gamma_3(0)$ and $\gamma_3(1)$.
    In this case, we can compute the stratification into singular loci of its algebraic boundary symbolically.
    The results of this computation, using \texttt{Macaulay2} \cite{M2}, are displayed in Table \ref{tab:alg_bd_strata}.
    \begin{table}[h]
        \centering
        \begin{tabular}{c|c|c|c|c}
                codim & 1 & 2 & 3 & 4 \\ \hline\hline
                \multirow{4}{4em}{\centering strata} & $\operatorname{CH}(C_3)$ & $\operatorname{Sing}(\operatorname{CH}(C_3))$ & $L_0$ & $\Sec$ \\
             & $\CH(\Sec)$ & $S_0$ &  $L_1$ & $\mathbf{T}_{\gamma_3(0)}(C_3)$ \\ 
             & & $S_{1}$ & $C_0$ & $\mathbf{T}_{\gamma_3(1)}(C_3)$ \\ 
             & & $S_{01}$ & $C_1$ & \\ \hline
             \end{tabular}
        \caption{Strata of the algebraic boundary of $\aA^\infty_2\subset \gr(2,4)$ ordered by codimension.}
        \label{tab:alg_bd_strata}
    \end{table}
    We can interpret each stratum geometrically, see  Section~\ref{sec:stratification}. The irreducible components are two Chow hypersurfaces: Firstly, the variety of lines in $\PP^3$ that intersect $C_3$, that is $\CH(C_3)$, secondly the variety of lines in $\PP^3$ intersecting the secant line $\Sec$, that is $\CH(\Sec)$.
    It is then well known that the singular locus of the variety of lines meeting the twisted cubic, is precisely the variety of secant lines (i.e. those lines that meet the twisted cubic in two points), which is of codimension $2$.
    We have $3$ more strata in this codimension which come from the intersection of the codimension $1$ strata: $S_0$ and $S_1$ are the surfaces of lines going through the points $\gamma_3(0)$ and $\gamma_3(1)$, respectively.
    $S_{01}$ is the surface of lines intersecting both $\Sec$ and $C_3$, generically in one point each. 
    In codimension $3$, which is dimension $1$, we have four irreducible components: $L_0$ and $L_1$ are the lines in $\gr(2,4)$ consisting of lines containing the point $\gamma_3(0)$, $\gamma_3(1)$ and intersecting
    \[
    \operatorname{rowspan}
    \begin{pmatrix}
        \gamma_3(1) \\ \partial_t\gamma_3(0)
    \end{pmatrix} \quad \text{and} \quad \operatorname{rowspan}
    \begin{pmatrix}
        \gamma_3(0) \\ \partial_t\gamma_3(1)
    \end{pmatrix},
    \]
    respectively, where $\partial_t$ means taking the derivative with respect to $t$.
    On the other hand, $C_0$ and $C_1$ are the conics in $\gr(2,4)$ of secants going through the points $\gamma_3(0)$ and $\gamma_3(1)$.
    Finally, the three vertices are precisely the tangent lines $\mathbf{T}_{\gamma_3(0)}(C_3)$ and $\mathbf{T}_{\gamma_3(1)}(C_3)$, as well as the secant line $\Sec$ itself, all being part of the singular locus of the $1$ dimensional strata. 
\end{example}

\begin{remark}\label{rem:bezout_matrix}
By Theorem \ref{thm:algebraic_boundary_of_A_infty} there are two Chow hypersurfaces of interest to us $\CH(C_{k+1})$ and $\CH(\Sec)$. The latter, has a rather simple structure: The intersection condition, $L\cap \Sec$, gives us a linear condition on the Pl\"uckers of $\gr(k,k{+}2)$, and is thus a hyperplane section of the Grassmannian. The Chow hypersurface of the rational normal curve, on the other hand, is the vanishing locus of the determinant of the B\'ezout matrix, as we shall see now. Start by considering the homogeneous parameterisation of the rational normal curve
\[
     C_{k+1}=\{[s^{k+1}:s^{k}t:\ldots:t^{k+1}]\mid[s:t]\in \PP^1\} \subset\PP^{k+1}.
\]
Also, consider a codimension $2$ plane $H\in\PP^k+1$ given as the intersection of two codimension $1$ hyperplanes, i.e.
\begin{align}\label{eq:H_kernel}
    H = \ker\begin{bmatrix}
    x_0 & \dots & x_{k+1}\\
    y_0 & \dots & y_{k+1}
    \end{bmatrix},
\end{align}
where $[x_0:\ldots:x_{k+1}],[y_0:\ldots:y_{k+1}]\in\PP^{k+1}$.
We then define two univariate polynomials in $t$ as
\begin{align}\label{eq:resultant_poly}
f=x_0+x_1\cdot t+\ldots+x_{k+1}\cdot t^{k+1}\quad\text{and}\quad
g=y_0+y_1 t+\ldots+y_{k+1}\cdot t^{k+1},
\end{align}
where the coefficients live in $\C$.
Then, $H$ intersects $C_{k+1}$ in $\gamma_{k+1}(t)$, that is $H\in\CH(C_{k+1})$, if and only if $f(t) = g(t) = 0$, that is to say when the two polynomials have a common root in $\C$. This, in turn, is precisely the case when the resultant of $f$ and $g$ vanishes. The resultant is the determinant of the Sylvester matrix of $f$ and $g$, which equals the determinant of the \emph{B\'ezout matrix} $B(f,g)$. For $0\leq i,j\leq k$, let $m_{ij}=\min(i,n-1-j)$ and $M_{ij}=\max(0,i-j)$, then we can define the entries of the complex $(k+1)\times(k+1)$ B\'ezout matrix as
\[ B_{ij}=\sum_{k=M_{ij}}^{m_{ij}}p_{(j+k+1),(i-k)},
\]
where the $p_{(j+k+1),(i-k)}$ are the primal Pl\"ucker coordinates of $H$. These are given by the $2\times 2$ minors of the matrix representing $H$ in Equation \eqref{eq:H_kernel}, obtained by deleting all columns except those corresponding to the subscript. This shows that $\CH(C_{k+1})$ is given by the vanishing locus of $\det B(f,g)$. 
\end{remark}

The proof of Theorem \ref{thm:algebraic_boundary_of_A_infty} proceeds via a series of Lemmas, beginning with a direct generalization of \cite[Lemma 2.2]{ranestad2024adjoints}.
\begin{lemma} \label{lem:projaway}
    Fix a partition $I_n=(0=t_1<\ldots<t_n=1)$ of the interval $[0,1]$,
    and let $Z$ be the totally positive $ n\times(k+2)$ matrix with rows $\gamma_{k+1}(t_1)$ through $\gamma_{k+1}(t_n)$.
    Let $A\in {\gr}(k,k{+}2)$ be a real $k$-space in $\R^{k+2}$ given by the matrix with rows $A_1$ through $A_k$. 
    The projection  $\R^{k+2} \to \R^2$ away from $A$ is represented by a matrix $A^\perp \in \mathbb{R}^{2 \times (k+2)}$, whose rowspan is the orthogonal complement of the rowspan of $A$.
    Let $z_i = A^\perp Z_i \in \mathbb{R}^2$, where $Z_i \in \mathbb{R}^{k+2}$ is the $i$-th row of $Z$, and let $M = (z_1,z_2,\ldots,z_n) \in \mathbb{R}^{2 \times n}$.
    There is a nonzero constant $c \in \mathbb{R}$ such that for all $1 \leq i < j \leq n$ the $(i,j)$ minor of $M$ is \[ \det(M_{ij}) = \det(z_iz_j) = c \cdot \langle A_1\ldots A_kij \rangle.\]
\end{lemma}
\begin{proof}
The matrix $A\in \R^{k\times (k+2)}$ has rank $k$. Therefore, its transpose $A^{T} \in \R^{(k+2)\times k}$ has a left inverse $A^\dag \in \R^{k\times (k+2)}$ so that $A^\dag A^{T} = {\operatorname{id}}_{k\times k}$. Block multiplication of matrices shows
\[ \begin{pmatrix}
    A^\dag \\ 
    A^\perp
\end{pmatrix} \cdot \begin{pmatrix}
    A^{T} ~ Z_i ~  Z_j
\end{pmatrix} \, =  \, \begin{pmatrix}
    {\operatorname{id}}_{k \times k} & A^\dag(Z_i~Z_j) \\ 
    0 & M_{ij}
\end{pmatrix}. \]
Taking determinants gives $\det(M_{ij})$ on the right hand side and 
\[ \det \begin{pmatrix}
    A^\dag \\ 
    A^\perp
\end{pmatrix} \cdot \langle A_1\ldots A_kij \rangle
\] 
on the left hand side. Since the image of $A^\dag$ is orthogonal to $A^\perp$, the constant $c = \det \begin{pmatrix}
    A^\dag \\ 
    A^\perp
\end{pmatrix}$ is nonzero.
\end{proof}

Similarly there is an immediate generalization of \cite[Theorem 3.1]{ranestad2024adjoints}.
\begin{lemma}\label{lem:geom_alg_bd}
    Let $n>k+2$. The algebraic boundary $\partial_a\mathcal{A}_{n,k}$ of $\mathcal{A}_{n,k}$ is the union of the hyperplane sections of $\gr(k,k{+}2)$ given by $\langle A_1\dots A_k ii{+}1\rangle=0$ for $i\in[n{-}1]$, and $\langle A_1\dots A_k 1n\rangle=0$, where $A_i\in\mathbb{P}^n$.
    Geometrically speaking there is a Zariski dense set of spaces $A\in \gr(k,k{+}2)$, given as a rank $k$ matrix with rows, $A_1$ through $A_k$ in the boundary stratum $\partial_a\mathcal{A}_{n,k,2}\cap \langle A_1\dots A_k i i{+}1\rangle$ which intersect the convex hulls $\operatorname{conv}(Z_{i},Z_{i+1})$ and $\operatorname{conv}(Z_{i+1},Z_{i+2},Z_{i+3})$ for all $i\in [k]$ read cyclically.
\end{lemma}
\begin{proof}
    We follow \cite{ranestad2024adjoints}.
    Set $i=1$ and write $e_i$ for the $i$-th standard basis row-vector of $\R^n$.
    Consider the following $2k{-}1$ dimensional semialgebraic set of rank $k$ matrices $\mathcal{Y}\subset \gr(k,n)_{\geq 0}$ with non-negative maximal minors, comprised of $k\times n$ matrices $Y$ whose rows $Y_j$ are given as \begin{align}\label{eq:family_of_NN_mat}
     Y_j = 
    \begin{cases}
       e_1+ y_{1,1}e_2 & j=1,\\
       y_{i,1}e_{i{+}2}+y_{i,2}e_{i{+}3}+e_{i{+}3} & k\geq j>1,
    \end{cases}
    \end{align}
    where $y_{i,j}\in\R_{\geq 0}$. 
    Since the intersection $\partial_a\mathcal{A}_{n,k}\cap \langle A_1\dots A_k 1 2\rangle$ is irreducible and of dimension $2k{-}1$ we get that $\widetilde{Z}(\mathcal{Y})$ is Zariski dense in $\partial_a\mathcal{A}_{n,k}\cap \langle A_1\dots A_k 1 2\rangle$. 
    For a generic $Y\in \mathcal{Y}$, denote its image under $\widetilde{Z}$ in $\gr(k,k{+}2)$  by $A=(A_1\dots A_k)$.
    Geometrically speaking, the space $A\in\gr(k,k{+}2)$ given by a rank $k$ matrix with rows $A_1$ through $A_k$, intersects $\operatorname{conv}(Z_1,Z_2)$ and $\operatorname{conv}(Z_i,Z_{i{+}1},Z_{i{+}2})$ for all $i\in [k]\setminus\{1\}$.
    It follows that $\langle A_1\ldots A_k12\rangle=0$ as well as $\langle A_1\ldots A_kij\rangle\neq 0$ for other $i<j\in [n]$.
    To get the other cases, $i\neq 1$, define the operator $\sigma:\mathbb{R}^{k\times n}\rightarrow\mathbb{R}^{k\times n}$ which applies a circular shift to the columns and then flipping the signs in the first column.
    After that, simply consider $\sigma^{i-1}(\mathcal{Y})$.
    It remains to show that there are no other boundary strata, that is $\langle A_1\ldots A_k1j\rangle\neq 0$ for $j\in[n]$ for a space $A\in \aA^{n,k}$ whose rows are given by $A_1$ through $A_k$. This is immediate from the proof of \cite[Theorem 3.1]{ranestad2024adjoints} and Lemma \ref{lem:projaway}.
\end{proof}

\begin{lemma}\label{lem:limit_of_boundary_strata}
     For any $t\in[0,1]$ consider the subset $\mathcal{V}_t$ of $\CH(C_{k+1})\subset \gr(k,k{+}2)$ consisting of $k$-spaces containing the point $\gamma_{k+1}(t)$ of the rational normal curve $C_{k+1}$. Then, the Euclidean boundary of $\aA^\infty_k$ intersects $\cV_t$ in a Zariski dense subset of $\mathcal{V}_t$. 
\end{lemma}
\begin{proof}
    Fix $n_0>k+2\in \N$, a partition $I_{n_0}$ of $[0,1]$ and $t\in[0,1]$ such that $t_i\leq t\leq t_{i+1}$ for some $i<n_0$.
    First, consider the case where we set $i=1$.
    Let $\mathcal{Y}$ be the $(2k{-}1)$-dimensional family of $k\times n$ matrices in $\gr(k,n)_{\geq 0}$ from the proof of Lemma \ref{lem:geom_alg_bd} with the shifting operator $\sigma$.
    $\mathcal{Y}$ was comprised of $k\times n$ matrices $Y$ whose rows $Y_j$ are given as \begin{align}
     Y_j = 
    \begin{cases}
       e_1+ y_{1,1}e_2 & j=1,\\
       y_{j,1}e_{j{+}2}+y_{j,2}e_{j{+}3}+e_{j{+}3} & k\geq j>1,
    \end{cases}
    \end{align}
    where $y_{i,j}\in\R_{\geq 0}$ and we consider $e_i$ as the standard basis row-vector of $\R^n$.
    Now take a sequence of partitions $\mathcal{I}= \{I_n\}_{n\geq n_0}$ of $[0,1]$ with $I_{n_0}<I_n<I_{n{+}1}$ in refinement order for all $n$ and where we insert $s_{i_n}$ between the original $t_1$ and $t_{2}$ of the partition $I_{n_0}$, while keeping the other $t_j$'s unchanged. For any $I_n\in\mathcal{I}$ we can define a family $\mathcal{Y}_n$ as above.
    Therefore, we can always find $s_{1},s_{2}\in I_n $ such that $s_{1}\leq t\leq s_{2}$. In turn, this implies that for any $Y_n\in\mathcal{Y}_n$ the pointwise limit, $\lim_{n\rightarrow\infty} Y_nZ$, converges to a $k$-space $A\in\gr(k,k{+}2)$ intersecting $C_{k+1}$ in $\gamma_{k+1}(t)$. 
    This space will also intersect
    the $\operatorname{conv}(\gamma_{k+1}(t_{j}),\gamma_{k+1}(t_{j+1}),\gamma_{k+1}(t_{j+2}))
    $ for all $j\in[k{-}1]+1$ by construction.
    For the cases where $t\in[0,1]$ such that $t_{i}\leq t\leq t_{i+1}$ for $i>1$, we simply use the shift $\sigma$ and consider $\sigma^{i-1}Y$ for $Y\in\mathcal{Y}$.
    Then, using the same family of partitions $\mathcal{I}$ and family of $k\times n$ matrices $\mathcal{Y}_n$, where $(\sigma^{i-1}Y)_n=\sigma^{i-1}Y_n$, we get $A=\lim_{n\rightarrow\infty} (\sigma^{i-1}Y)_nZ \in \gr(k,k{+}2)$ which intersects $C_{k+1}$ in $\gamma_{k+1}(t)$.
    Denote the set of all of these point wise limits $\mathcal{Y}_t\subset\gr(k,k{+}2)$, then we have just shown that $\mathcal{Y}_t\subset \mathcal{V}_t$.
    As $\operatorname{codim}(\widetilde{Z}(\mathcal{Y}))=2$ which equals $\operatorname{codim}(\mathcal{V}_t)=2$, the set $\mathcal{Y}_t$ is Zariski dense in $\mathcal{V}_t$, finishing the proof.
\end{proof}

Equipped with the results from Lemmas \ref{lem:geom_alg_bd} and \ref{lem:limit_of_boundary_strata} we can now prove the main Theorem of this section.

\begin{proof}[Proof of Theorem \ref{thm:algebraic_boundary_of_A_infty}]
    We use the notation from Lemma \ref{lem:limit_of_boundary_strata}.
    Since $\aA_k^\infty$ is the union over all partitions of $[0,1]$, the following map is well defined
    \[
    \varphi:[0,1]\longrightarrow\aA_k^\infty\subset\gr(k,k{+}2)\quad,\quad t\mapsto \mathcal{Y}_t\subset \mathcal{V}_t.
    \]
    For any fixed $t\in[0,1]$ we therefore get an irreducible semialgebraic set $\mathcal{Y}_t$ of codimension $1$ in $\CH(C_{k+1})$.
    Since the domain of $\Psi$ is $1$-dimensional we get a top-dimensional irreducible semialgebraic subset in $\CH(C_{k+1})$. Thus, the image $\varphi([0,1])$ is Zariski dense in $\CH(C_{k+1})$, more succinctly $\partial_a\aA_k^\infty\supset\CH(C_{k+1})$.
    Finally, for any $I_n$ of $[0,1]$ it remains to consider the remaining boundary stratum $\langle A_1\ldots A_k\gamma_{k+1}(1)\gamma_{k+1}(n)\rangle=0$ of the algebraic boundary $\partial_a\aA_{I_n}$, as by Lemma \ref{lem:geom_alg_bd}.
    However, as $t_1=0$ and $t_n=1$ for all partitions, we get $\Sec$ spanned by $\gamma_{k+1}(0)$ and $\gamma_{k+1}(1)$. Therefore, $\partial_a\aA_k^\infty=\CH(C_{k+1})\cup \CH(\Sec)$, where we identified $\langle A_1\ldots A_k\Sec\rangle=0$ with $\CH(\Sec)$ as by Remark \ref{rem:physics_lingo}.
\end{proof}

\section{Stratification}\label{sec:stratification}

The purpose of this section is to systematically study the singularity structure of the algebraic boundary $\partial_a\aA^\infty_k$. To that end, we use tangent spaces of incident varieties to study iterated singularities of Chow hypersurfaces.
In Section \ref{sec:tangent_spaces}, we focus on the relevant incidence varieties. In Section \ref{sec:strata}, we apply these tools to compute the iterated singular loci and give the relevant boundary stratification from the point of view of Positive Geometry. Note, in this section we have $d=k+1$ in an effort to simplify notation.

\subsection{Tangent spaces and incidences}\label{sec:tangent_spaces}

Our discussion here is based on \cite[Chapter 14]{Harris_AG}.
Let $X$ be a smooth irreducible projective variety in $\PP^{d}$ defined by its prime ideal $I=(F_1,\dots,F_r)$, generated by homogeneous polynomials $F_i$ in $\mathbb{C}[z_0,\dots,z_n]$ for all $1\leq i\leq r$. We denote by $\mathbf{T}_xX$ the \emph{projective tangent space} at $x\in X$, defined as the kernel of the Jacobian $(\partial F_i/\partial z_j)_{ij}$ of $I$ at $x$. For the Grassmannian $\gr(k,n)$, we can identify its tangent space at $V$ with $\operatorname{Hom}(V,\mathbb{C}^{n}/V)$; to make this identification obvious, we will use a non-bold typeface, i.e. $T_V\gr(k,n)=\operatorname{Hom}(V,\mathbb{C}^{n}/V)$. 
In particular, for any $p\in \PP^n$, we get $T_p\PP^n = \Hom(p,\C^{n+1}/p)$, where we think of $p$ as a $1$-dimensional subspace of $\C^{n+1}$.
For the remainder of this section we are interested in the case where $X$ is the rational normal curve $C_d\subset \PP^d$ of degree $d$.

\begin{definition}\label{def:divisor}
    Let $C\subset \PP^d$ be a curve and $\CH(C)\subset \gr(d-1,d{+}1)$ be its Chow hypersurface. To a space $V\in \ch(C)$, we associate a divisor $D_V\in \Div(C)$ defined as the intersection divisor $H.C$ for a general hyperplane $H\subset \PP^d$ containing $V$. Concretely, if $H$ is defined by the linear form $\ell$, then the restriction of $\ell$ to $C$ is a section of the line bundle $\sO(1)\vert_C$ and $D_V$ is the (zero) divisor of this section. The \emph{degree} $\deg(D)$ of a divisor $D = \sum_{p\in C} m_p p$ equals $\Sigma_{p\in C}m_p$.
    We write $S^\ell$ for the symmetric $\ell$-fold product $C^\ell/S_\ell$, the quotient of the Cartesian product modulo the action of the symmetric group $S_\ell$ acting by permutation on the factors. 
\end{definition}

\begin{remark}
    The divisor $D_V\in \Div(C)$ that we just defined is equal to $D_V = \sum_{p\in C} m_p p$, where $m_p$ is the smallest order of vanishing of any linear form $\ell$ vanishing on $V$. 
\end{remark}

For any $\ell\in \N$, we denote by $\secV^\ell(C_d)$ the \emph{variety of $\ell$ secants} in $\grd$, which consists of all subspaces intersecting $C_d$ in $D_V$ such that $\deg(D_V)=\ell.$
Note $\secV^\ell(C_d)=\varnothing$ for any $\ell>d{-}1.$
Also recall that points in the $\ell$-fold symmetric product $S^\ell(C_d)$ are in correspondence with effective divisors on $C_d$ of degree $\ell$, therefore 
we can then define the incidence
\begin{align}\label{eq:incidence}
    \Sigma_\ell=\left\{(p_1+\ldots +p_\ell,V)\:\bigg|\:  \sum_{i=1}^\ell p_i \leq D_V \right\}\subset S^\ell(C_d)\times \gr(d-1,d+1),
\end{align}
with the canonical projections $\pi_1$ and $\pi_2$ into the first and second factor, respectively.
Then we have $\pi_2\left(\Sigma\cap\pi_1^{-1}( S^\ell(C_d))\right)=\secV^\ell(C_d)$. 
Following \cite[Chapter 14]{Harris_AG}, we describe the tangent space to $\Sigma_\ell$ in terms of complex analytic geometry. 
Any tangent vector to $\Sigma_\ell$ at $(p_1+\ldots+p_\ell,V)$ is equal to $\sigma'(0)$, as an element in the tangent space of $ S^\ell(C_d)\times\grd$, i.e.
$\prod_{i=1}^\ell\operatorname{Hom}(p_i,\C^{d+1}/p_i)\times \operatorname{Hom}(V,\C^{d+1}/V)$, for a suitable holomorphic arc $\sigma(t) = (p_1(t)+\ldots+p_\ell(t),V(t))$.
Here, $p_i(t)$ is a holomorphic arc in $C_d$ for each $i\leq \ell$ and $V(t)$ is a holomorphic arc in $\grd$ such that $p_i(t)\in V(t)$ for all $t$ and $i\leq\ell$, and with $p_i(0)=p_i$ and $V(0)=V$. Write $\sigma'(0) = (\theta_1,\ldots,\theta_\ell,\phi)$.
First consider a subspace $V\in\secV^\ell(C_d)$ that meets $C_d$ in $\ell$ distinct points. Then the containment $p_i(0)=p_i\in V$  implies that the restrictions of $\phi$ to the $1$-dimensional linear space $p_i$ must coincide with $\theta_i$ modulo $V$ i.e.  $\phi\vert_{p_i}=\theta_i+V$ for all $i\leq \ell$.
Therefore, the tangent space of $\Sigma$ is
\begin{align}\label{eq:incidence_tangent_space_generic}
    T_{(p_1+\ldots+p_\ell,V)} \Sigma_\ell=\{(\theta_1,\ldots, \theta_\ell,\varphi)\mid \varphi\vert_{p_i}=\theta_i + V \text{ for all } i\leq\ell\},
\end{align}
as a subspace of $\prod_{i=1}^\ell\operatorname{Hom}(p_i,\C^{d+1}/p_i)\times \operatorname{Hom}(V,\C^{d+1}/V).$
Moreover, we also get an intrinsic characterization of the tangent space of $\secV^\ell(C_d)$ at points with $p_i\neq p_j$ ($i\neq j$) as 
\begin{align}\label{eq:secant_tangent_space}
    T_V\secV^\ell(C_d)=\{\varphi \mid \varphi(p_i)\subset \mathbf{T}_p(C_d)+V \text{ for all } i\leq\ell\}\subset \operatorname{Hom}(V,\C^{d+1}/V).
\end{align}

Next, we deal with multiplicity of intersection points: For $V\in\grd$ with $D_V=\sum_{i=1}^r m_ip_i$ such that $\deg(D_V)=\ell$, its preimage under $\pi_2$ is $(m_1 p_1 + m_2 p_2+\ldots +m_r p_r,V)\in\Sigma_\ell$, where each $p_i$ appears exactly $m_i$ times.
Geometrically speaking $m_i>1$ means that $V$ will contain the $m_i$-th osculant plane $\mathcal{O}^{(m_i)}(p_i)$. In order to define the $m$-th osculant plane, consider the \emph{$m$-th Gauss map} 
\[
\mathcal{G}^{(m)}:C_d\longrightarrow \gr(m,d+1)
\]
that takes $p\in C_d$ to the row span of the matrix of higher derivatives  $(v,v^{(1)},\ldots,v^{(m-1)})$, where $v(t)$ is an holomorphic arc around $p\in C_d$, and $v^{(i)}$ denotes the $i$-th derivative of that arc. 
The image of $p\in C_d$ under $\mathcal{G}^{(m)}$ is then called the \emph{$m$-th osculant plane} $\mathcal{O}^{(m)}(p)$.
Therefore, for such a $V$ as above, the containment $\sum_{i=1}^\ell p_i\leq D_V$ implies that $\mathcal{O}^{(m_i)}(p_i)\subset V$, which in turn implies that $\varphi\vert_{\mathcal{O}^{(m_i)}(p_i)}=\theta_{i}+V$, where $\theta_{i}$ is the tangent vector of the holomorphic arc $v_i(t)$ in $\gr(m_i,d{+}1)$ with $v_{i}(t)\subset V(t)$ for all $t$ and $v_i(0)=\mathcal{O}^{(m_i)}(p_i)$.
Thus, the tangent space to $\Sigma_\ell$ at the unique point in $\pi_2^{-1}(V)=(D,V)$ with $D = \sum_{i=1}^r m_i p_i$ is
\begin{align}\label{eq:tangent_space_higher_order}
    T_{(D,V)}\Sigma_\ell=\{ {(\theta_1,\ldots,\theta_r,\varphi)\mid \varphi\vert_{\mathcal{O}^{(m_i)}(p_i)}=\theta_i+V \text{ for all } i\leq r}\},
\end{align}
where the vectors $\theta_i$ are in $\Hom(\mathcal{O}^{(m_i)}(p_i),\C^{d+1}/\mathcal{O}^{(m_i)}(p_i))$.

The linear spans of multiple osculating planes to the rational normal curve always have the expected dimension by a generalization of Vandermonde matrices. 
\begin{remark}\label{rem:generalized_vandermond}
Fix a partition $I_n$ of $[0,1]$ and $r=(r_1,\ldots,r_n)\in\N^n$ with $r_1 + \ldots + r_n = d+1$. We define the \emph{generalized Vandermonde matrix} $V_r \in \C^{(d+1)\times (d+1)}$ as
\[
V_r=(\gamma^{(0)}_d(t_1),\ldots, \gamma^{(r_1)}_d(t_1),\ldots,\gamma^{(0)}(t_n),\ldots, \gamma^{(r_n)}_d(t_n)),
\]
where $\gamma_d^{(k)}(t)$ denote the $k$-th derivative of $\gamma_d(t)$.
Geometrically, speaking this Vandermonde matrix is constructed from the $r_i$-osculant planes $\mathcal{O}^{(r_i)}(\gamma_d(t_i))$ for all $i\leq n$.
The determinant of this matrix is
\[
\label{eq:generalized_vandermonde}
    \det(V_r) = \prod_{i=1}^n(1!\ldots r_i!)\prod_{j>i}(t_j-t_i)^{(r_j+1)(r_i+1)},
\]
see e.g. \cite{generalized_van_der_monde}. 
In particular, $\det V_r=0$ if and only if $t_j=t_i$ for some $i\neq j$. This result implies that osculant planes to distinct points of the rational normal curve are transversal (and in particular do not intersect as long as $\sum m_i < d+1$).
\end{remark}

\begin{proposition}\label{prop:smooth_tangent}
    For $\ell\leq d{-}1$, the incidence $\Sigma_\ell$  from \eqref{eq:incidence}  is smooth of codimension $2\ell$.
\end{proposition}

\begin{proof}
Fix $\ell\leq d-1$. 
Let $\pi_1$ be the projection from the incidence $\Sigma_\ell$ into the first factor $ S^\ell(C_d)$ and $\pi_2$ the projection into the Grassmannian.
First, we determine the dimension of $\Sigma_\ell$.
The variety $ S^\ell(C_d)\times \gr(d{-}1,d{+}1)$ has dimension $\ell+2(d-1)$.
Fix $D = p_1+\ldots + p_\ell \in S^\ell(C_d)$ and consider 
\[
\mathcal{V}_D=\left\{V\in\gr(d{-}1,d{+}1)\:\bigg|\: \sum_{i=1}^\ell p_i\leq D_V\right\}.
\]
Denote by $O$ the complex vector space spanned by all the osculant planes $\mathcal{O}_i=\mathcal{O}^{(m_i)}(p_i)$; it then follows from the generalized Vandermonde determinant \eqref{eq:generalized_vandermonde}, that $\dim(O) = \ell$.
Therefore, we have that $\mathcal{V}_D\cong \gr(d{-}\ell{-}1,d{-}\ell{+}1)$, which is $2(d-1)-2\ell$ dimensional.
Since $D$ varies in an $\ell$-dimensional variety, we conclude $\dim(\Sigma_\ell)=\dim(\mathcal{V}_D)+\ell$; thus
$\operatorname{codim}(\Sigma_{\ell})=2\ell$.

To show that $\Sigma_\ell$ is smooth recall the notation for the tangent space \eqref{eq:tangent_space_higher_order} for $V\in\grd$ with $D_V=\sum_{i=1}^rm_ip_i$, all $p_i$ distinct and $\deg(D_V)=\ell$:
\[
T_{(D,V)}\Sigma_\ell=\{ {(\theta_1,\ldots,\theta_r,\varphi)\mid \varphi\vert_{\mathcal{O}_i}=\theta_i+V\quad\forall i\leq r}\}.
\]
Then, the condition $\mathcal{O}_i\subset V$ for the incidence $\Sigma_\ell$ implies $\varphi\vert_{\mathcal{O}_i}=\theta_i + V$. 
Since the osculant planes are transveral (see Remark~\ref{rem:generalized_vandermond}) and $\sum m_i \leq d$, the conditions $\varphi\vert_{\mathcal{O}_i}=\theta_i+V$ on the linear map $\varphi$ are independent. To count the conditions for each $i$, we choose any basis of $\mathcal{O}_i$.
Then, $\varphi\vert_{\mathcal{O}_i}=\theta_i + V$ fixes the images of $\varphi(\mathcal{O}_i)$, that is to say, it fixes $m_i$ columns of the $2\times (d{+}1)$ matrix ($\dim(\C^{d+1}/V) = 2$) representing $\varphi$ with respect to the chosen basis. Thus placing $2m_i$ linear conditions on elements in $T_{(D,V)}\Sigma_\ell$. In total we therefore have $2\ell$ linear constraints for a linear map $\varphi \in \Hom(V,\C^{d+1}/V)$ to be in $T_{(D,V)}( S^\ell(C_d)\times \grd)$.
Since we have shown above that $\operatorname{codim}(\Sigma_\ell)=2\ell$, and since $D_V$ was arbitrary, we proved that $\Sigma_\ell$ is smooth everywhere.
\end{proof}

\subsection{Strata}\label{sec:strata}

We now turn to the study of the strata of the algebraic boundary $\partial_a\aA^\infty_{d-1}$.
Recall that the variety of $\ell$-secants $\secV^\ell(C_d)$ was the subset of $\grd$ consisting of planes $V$ with $\deg(D_V)\geq\ell$.
Moreover, for a projective variety $X$ we denote its \emph{singular locus} by $\Sing(X)$, which can be computed via a Jacobian criterion, see \cite[Chapter 10]{IdealsVarieties} for details.
We start by considering the irreducible component of $\partial_a\aA^\infty_{d-1}$ which stems from the Chow hypersurface of the rational normal curve $C_{d}$.

\begin{proposition}\label{prop:it_sing_locus_chow}
    For any $\ell$ in $\N$ the singular locus of $\secV^\ell(C_d)$ is given by
    \[
    \operatorname{Sing}(\secV^\ell(C_d))=\secV^{\ell+1}(C_{d}).
    \]
\end{proposition}

\begin{proof}
    Fix $\ell < d-1$.
    We recall the setup for the incidence $\Sigma_\ell$ from \eqref{eq:incidence} and its tangent space for ease of readability:
    \[
    \Sigma_\ell=\left\{(p_1+\ldots +p_\ell,V)\:\bigg|\:  \sum_{i=1}^\ell p_i\leq D_V\right\}\subset  S^\ell(C_d)\times \gr(d{-}1,d{+}1)
    \]
    We have, by construction
    \[
    \operatorname{Sec}^{\ell}(C_d)=
    \pi_2\left(\Sigma_\ell\cap \pi_1^{-1}\left( S^\ell(C_d)\right)\right).
    \]
    Also recall that by \eqref{eq:tangent_space_higher_order} the tangent space of $\Sigma_\ell$ at $(D,V)$ for $V\in\grd$ such that $D_V=\sum_{i=1}^rm_ip_i = D$ and $\deg(D_V)=\ell$ and $\pi_2^{-1}(V)=(D,V)$ is given as 
    \[
    T_{(D,V)}\Sigma_\ell=\{ {(\theta_1,\ldots,\theta_r,\varphi)\mid \varphi\vert_{\mathcal{O}^{(m_i)}(p_i)}=\theta_i+V\quad\forall i\leq r}\},
    \]
        where the $\theta_i$ are elements in $\Hom(\mathcal{O}^{(m_i)}(p_i),\C^{d+1}/\mathcal{O}^{(m_i)}(p_i))$.
    Then, as the symmetric product $S^\ell(C_d) \cong \PP^\ell$ of the rational normal curve $C_d$ is smooth, see \cite[Proposition 10.6]{and_all_that},
    and as we further assume that all $p_i$ are distinct, we get a unique smooth point in $\Sigma_{\ell}\cap \pi_1^{-1}\left( S^\ell(C_d)\right)$ as the preimage over $V$.
    It therefore follows that
    \[\pi_2:\Sigma_{\ell}\cap\pi_1^{-1}( S^\ell(C_d))\longrightarrow \secV^{\ell}(C_d)\subset \gr(d{-}1,d{+}1)
    \]
    is one-to-one over $(D,V)$.
    To characterise the smooth points, we want to make use of \cite[Corollary 14.10]{Harris_AG}, so we need to understand when the differential map induced by $\pi_2$ is injective.
    Observe that we just showed that if $(\theta_1,\ldots,\theta_\ell,0)\in T_{(D,V)}\Sigma_\ell$, then by $\varphi\vert_{\mathcal{O}^{(m_i)}(p_i)}=\theta_i+V$, we must have $V\subset \mathcal{O}^{(m_i)}(p_i)$, this gives us a contradiction by dimensions.
    Therefore, $\pi_2$ has injective differential, and it follows from \cite[Corollary 14.10]{Harris_AG} that $\secV^{\ell}(C_d)$ is smooth at $V$, as chosen above.
    It remains to show that every smooth point of $\operatorname{Sing}^{\ell}(\CH(C_d))$ is of this form. To that end, we employ Zariski's weak main theorem, \cite[Proposition 16.8]{Harris_AG}.
    Then, it suffices to note that any point $W$ in $\Sing^{\ell-1}(\CH(C_d))$ such that $\deg(D_W)>\ell$ will have at least two distinct (disconnected) preimages $\pi_2^{-1}(W)$, and thus be a singular point. Since $\ell$ and $D_V$ were arbitrary this finishes the proof.
    If $\ell\geq d-1$, then $\secV^\ell(C_d)$ is smooth and as $\secV^{\ell+1}(C_d)=\varnothing$, the claim follows immediately for all $\ell$.
\end{proof}

In the proof of the preceeding Proposition~\ref{prop:it_sing_locus_chow} we actually determine the tangent space to the variety of $\ell$-secants.
\begin{corollary}\label{cor:SecL_tangent_space}
    Let $V\in\grd$ with $D_V=\sum_{i=1}^rm_ip_i$ where all $p_i$ distinct and $\deg(D_V)=\ell$.
    Then, $\secV^\ell(C_d)$ is smooth at $V$ with
    \[
        T_V\secV^\ell(C_d)=\{\varphi\in\operatorname{Hom}(V,\C^{d+1}/V)\mid \varphi(\mathcal{O}^{(m_i)}(p_i))\subset \mathcal{O}^{(m_i+1)}(p_i)+V \quad\forall i\leq r\}. \qedhere
    \]
\end{corollary}

In light of Remark \ref{rem:bezout_matrix} we have another immediate consequence.
\begin{corollary}
    Let $f$ and $g$ be the univariate polynomials as in Equation \eqref{eq:resultant_poly}. Then, the $\ell$-th iterated singular locus $\Sing^\ell(\CH(C_d))$ is the set of codimension $2$ subspaces, where $f$ and $g$ have precisely $\ell+1$ common roots, counting with multiplicity.
\end{corollary}

In the analysis of the singular loci of the variety of $\ell$-secants, special lines will play a role. Let $i\leq d-1$. We fix the symmetric notations
\begin{align}\label{eq:L0_L1}
 L_{0,i} & = \operatorname{rowspan}
 \begin{pmatrix}
     \gamma_d(1) \\ \gamma_d^{(i)}(0)
 \end{pmatrix}\quad\text{and} \quad 
 L_{1,i} = \operatorname{rowspan}
 \begin{pmatrix}
     \gamma_d(0) \\ \gamma_d^{(i)}(1)
 \end{pmatrix},
\end{align}
where $\gamma_d^{(i)}(t)$ denotes the $i$-th derivative of $\gamma_d$ at $t$.
Define $O_{0,i}\subset \grd$ to be the $2(d{-}i{-}1)$ dimensional variety of $(d{-}1)$-spaces containing the osculant plane $\mathcal{O}^{(i)}(\gamma_d(0))$ and intersecting the line $L_{0,i}$. 
In terms of the variety $\mathcal{V}_{\mathcal{O}^{(i)}}$, the set of subspaces $V\in \grd$ containing $\mathcal{O}^{(i)}(\gamma_d(0))$, we have $O_{0,i}=\CH(L_{0,i})\cap \mathcal{V}_{\mathcal{O}^{(i)}}$.
For any $i<\ell\leq d-1$, the shorthand $O_{0,i}^\ell$ denotes the intersection $O_{0,i}\cap\secV^\ell(C_{d})$, which has dimension $2(d{-}1){-}\ell{-}i{-}1$.
Symmetrically, we write $O_{1,i}^\ell$. 
For $0<i,j<\ell$ and $i+j<\ell$ the intersection $O_{0,i}^\ell\cap O_{1,j}^\ell$
is the $2(d{-}1){-}\ell{-}j{-}i$ dimensional variety in $\grd$ consisting of $(d{-}1)$-spaces containing both $\mathcal{O}^{i-1}(\gamma_d(0))$ and $\mathcal{O}^{j-1}(\gamma_d(1))$.
Therefore, we denote this intersection by $O^\ell(i,j)$.
We can now study the remaining strata of $\partial_a\aA^\infty_k$.

\begin{proposition}\label{prop:sing_CH(C_D)_CH(01)}
For any $\ell\leq d-1$, the singular locus of the intersection $\secV^{\ell}(C_d)\cap \CH(\Sec)$ decomposes as follows
\[
\operatorname{Sing}(\secV^\ell(C_d)\cap \CH(\Sec))=\left(\secV^{\ell+1}(C_d)\cap\CH(\Sec)\right)\:\cup \bigcup_{i=1}^\ell O^\ell_{0,i}\cup O^\ell_{1,i}.
\]
\end{proposition}
\begin{proof}
    For a projective variety $X\subset \PP^n$ and a hyperplane $H \subset \PP^n$ we have $p\in \Sing(X\cap H)$ if and only if $p\in\Sing(X)\cap H$ or $T_pX\subset T_pH$, see \cite[Chapter 9]{IdealsVarieties}.
    We apply this fact to $X=\secV^\ell(C_d)$ and $H=\CH(\Sec)$.
    So Proposition \ref{prop:it_sing_locus_chow} implies that the intersection $\secV^{\ell+1}(C_d)\cap\CH(\Sec)$ is contained in the singular locus. Moreover, the incidence 
    \[
    \Sigma = \left\{(s,p_1+\ldots+p_\ell,V)\:\bigg|\: s\in V\cap\Sec \land \sum_{i=1}^\ell p_i\leq D_V \right\}\subset \Sec\times S^\ell(C_d)\times \grd
    \]
    implies that this intersection is irreducible.
    Therefore, it remains to study the inclusion $T_V\secV^\ell(C_d)\subset T_V(\CH(\Sec))$ for $V\in\secV^\ell(C_d)\cap \CH(\Sec)$.
    Fix $V\in\grd$ with $D_V=\sum_{i=1}^rm_ip_i$ and $\deg(D_V)=\ell$, and abbreviate $\mathcal{O}_i=\mathcal{O}^{(m_i)}(p_i)$.
    The relevant tangent spaces are 
    \[
        T_V\secV^\ell(C_d)=\{\varphi\in\operatorname{Hom}(V,\C^{d+1}/V)\mid \varphi(\mathcal{O}_i)\subset \mathcal{O}_i+V \text{ for all } 1\leq i\leq r\}
    \]
    by Corollary~\ref{cor:SecL_tangent_space} and from \eqref{eq:secant_tangent_space}
    \[
        T_V\operatorname{CH}(\Sec)=\{\varphi\in\operatorname{Hom}(V,\C^{d+1}/V)\mid \varphi(q)\subset\mathbf{T}_q(\Sec)+V\},
    \]
    where $q$ is an intersection point of $V$ with $\Sec$.
    So we have containment $T_V\secV^\ell(C_d)\subset T_V\CH(\Sec)$ if and only if we have for the subspace
    \[
    \Lambda_V = \operatorname{span}_{\C}\{\varphi\in\operatorname{Hom}(V,\C^{d+1}/V)\mid \varphi(q)\subset\mathbf{T}_q(\Sec)+V \land \varphi(\mathcal{O}_i)\subset \mathcal{O}_i+V\quad\forall i \leq r\},
    \]
    that $\dim(\Lambda_V)=\dim(T_V\secV^\ell(C_d))$.
    Firstly, notice that $D_V$ dictates that the dimension of $\mathcal{O}_i\cap V$ equals $m_i$ for all $i$.
    Now suppose $q$ and the $p_i$'s are all distinct, then we can extend the $q$ and the corresponding osculant planes $\mathcal{O}_i$ to a basis of $V$. 
    The conditions $\varphi(O_i)\subset \mathcal{O}_i+V$ and $\varphi(q)\subset\Sec+V$ then fix, up to scaling, $\ell+1$ columns of the $2\times(d{-}1)$ matrix representing $\varphi$ with respect to the chosen basis.
    Therefore, 
    \[
    \dim(\Lambda_V)=\ell+1+2(d-\ell-2)=2(d-1)-\ell-1,
    \]
    which is strictly smaller than the dimension of $T_V\secV^\ell(C_d)$, which equals $2(d-1)-\ell$ by Proposition \ref{prop:smooth_tangent}.
    So it remains to consider $V$ such that $q=p_i$ for some $i$.
    In this case, we can again extend $q$ and the $p_i$'s to a basis of $V$. However, as $p_i=q$, the inclusions will only fix, up to scaling, $2\ell$ columns of the matrix representing $\varphi$ with respect to this basis.
    Thus, $\dim(\Lambda_V)=2(d-1)-\ell$, and thus in this case $V$ is a singular point.
    Since there are only two points which lie in $\Sec$ and $C_d$, i.e. $\gamma_d(0)$ and $\gamma_d(1)$, we only need to study these two cases explicitly.
    We start with $p=q=\gamma_d(0)$, where wlog. $p=p_1$.
    This implies $\gamma_d(0)\in V$ and $\mathcal{O}^{(m_1)}(\gamma_d(0))\subset V$ additionally we must have $\mathbf{T}_{\gamma_d(0)}(\Sec)+V\subset\mathcal{O}^{(m_1+1)}(\gamma_d(0))+V$. 
    This is the case if and only if the other basis element of $\mathbf{T}_{\gamma_d(0)}(\Sec)$, that is $\gamma_d(1)$ lies in the span of $\mathcal{O}^{(m_1+1)}(\gamma_d(0))$ mod $V$.
    This, in turn, is true if and only if $V$ contains a linear combination of $\gamma_d(1)$ and the basis vectors of $\mathcal{O}^{(m_1+1)}(\gamma_d(0))$. 
    For $m_1=1$, we have actually an equality condition, $\mathbf{T}_{\gamma_d(0)}(\Sec)+V=\mathbf{T}_{\gamma_d(0)}C_d+V$, such that $A = [0:s:t:\ldots:t]\in \PP^d$ must be contained in $V$ for some $[s:t]\in\PP^1$.
    Then it suffices to notice that this point $A$ lies on $L_{0,1}$ from \eqref{eq:L0_L1}.
    If $m_1>1$ and as $\mathcal{O}^{(m_1)}(\gamma_d(0))\subset V$, we can project away from $\mathcal{O}^{(m_1)}(\gamma_d(0))$, under this projection $\gamma_d^{(m_1)}(0)$ becomes $\gamma_{d-m_1}^{(1)}(0)$ and $\gamma_d(1)=\gamma_{d-m_1}(1)$.
    Therefore, again we find that $[0:s:t:\ldots:t]\in \PP^{d-m_1}$ must be contained in the projection of $V$ for some $[s:t]\in \PP^1$, which lifts to $[s:\ldots:s:t:\ldots:t]\in\PP^d$.
    Then, analogously notice that this point lies on $L_{0,m_1}$ from \eqref{eq:L0_L1}.
    Finally, we can use the incidence 
    \[
    \Sigma = \left\{(s,p_1+\ldots+p_\ell,V)\:\bigg|\: s\in V\cap L_{0,1} \land \sum_{i=1}^\ell p_i\leq D_V \right\}
    \subset L_{0,m_1} \times S^\ell(C_d)\times \grd
    \]
    to show that $O_{0,1}^\ell$ is irreducible.
    The reasoning for $p=q=\gamma_d(1)$ is identical. This finishes the proof. 
\end{proof}

\begin{proposition}\label{prop:sing_X_0l}
    For any $j<\ell\leq d-1$ we have
    \[
    \Sing(O^\ell_{0,j})=O^{\ell+1}_{0,j}\cup O^\ell_{0,j+1} \cup O^\ell(j,1),
    \]
    and symmetrically
    \[
    \Sing(O^\ell_{1,j})=O^{\ell+1}_{1,j}\cup O^{\ell}_{1,j+1} \cup O^\ell(1,j)
    \]
\end{proposition}
\begin{proof}
    Let $\pi_0:\C^{d+1}\rightarrow\C^{d-j+1}$ be the projection away from the the osculant plane $O = \mathcal{O}^{(j)}(\gamma_d(0))$, and let $\mathcal{V}_{O}=\{V\mid O\subset V\}\subset\gr(d{-}1,d{+}1)$. 
    Then, $\pi_0$ induces the regular map 
    \[
    \overline{\pi}_0:\mathcal{V}_{O}\subset\gr(d{-}1,d{+}1)\rightarrow\gr(d{-}j{-}1,d{-}j{+}1)
    \]
    which is, in fact, an isomorphism.
    Set $\overline{d}=d-j-1$.
    First, notice that $\pi_0(C_d)=(C_{\overline{d}})$, as we can recursively project away from $\gamma_d(0)$, under which $\gamma_{d}^{(1)}(0)$ becomes $\gamma_{d-1}(0)$, and then we repeat $(j{-}1)$-times.
    Secondly, we also get $\overline{\pi}_0(L_{0,j})=\overline{\Sec}\subset \PP^{\overline{d}}$, where $\overline{\Sec}$ is spanned by $\gamma_{\overline{d}}(0)$ and $\gamma_{\overline{d}}(1)$.
    We therefore get the following identification   
    \[
    O^\ell_{0,j}\cong\overline{\pi}_0 (\mathcal{V}_{O}\cap \secV^\ell(C_d)\cap \CH(L_{0,j}))
    =\secV^{\ell-j}(C_{\overline{d}})\cap \CH(\overline{\Sec}).
    \]
    We can now apply Proposition \ref{prop:sing_CH(C_D)_CH(01)} to find the singular locus:
    \[
    \Sing(\secV^{\ell-j}(C_{\overline{d}})\cap \CH(\Sec))
    =(\secV^{\ell-j+1}(C_{\overline{d}})\cap\CH(\overline{\Sec}))\cup \overline{O}^{\ell-j}_{0,1}\cup \overline{O}_{1,1}^{\ell-j},
    \]
    where the $\overline{O}^{\ell-j}_{0,1}$ is analogously defined as in $O^{\ell-j}_{0,1}$, but in the space $\PP^{\overline{d}}$. Same for $\overline{O}_{1,1}^{\ell-j}$.
    Then first notice that 
    \[
    \overline{\pi}_0^{-1} \left(\secV^{\ell-j+1}(C_{\overline{d}})\cap\CH(\overline{\Sec})\right)=O^{\ell+1}_{0,j}.
    \]
    Next, observe that any $\overline{V}\in \overline{O}_{1,1}^{\ell-j}$ contains a line $L$ spanned by $\gamma_{\overline{d}}(0)$ and $\gamma_{\overline{d}}(1)-s\gamma_{\overline{d}}^{(1)}(0)$ for some $s\neq 0 \in\C$.
    Then, $\gamma_{\overline{d}}(0)$ clearly lifts to $\gamma_{d}^{(j)}(0)$ under $\overline{\pi}_0$, and $\gamma_{\overline{d}}^{(1)}(0)$ to $\gamma_{d}^{(j+1)}(0)$.
    Therefore, the preimage $\overline{\pi}_0^{-1}(\overline{V})$ corresponds precisely to a $(d{-}1)$-space containing $\mathcal{O}^{(j+1)}(0)$ and a point on the line $L_{0,j+1}$, and thus $\overline{\pi}_0^{-1}(\overline{O}^{\ell-j}_{0,1})=O^\ell_{0,j+1}$.
    For $\overline{O}_{1,1}^{\ell-j}$ similarly observe that by $\gamma_{\overline{d}}^{(1)}(1)$ lifts to $\gamma_d^{(1)}(1)$ and $\gamma_{\overline{d}}(1)$ of course to $\gamma_d(1)$.
    Therefore, we can identify the preimage $\overline{\pi}_0^{-1}(\overline{O}_{1,1}^{\ell-j})$ with the subset of $\secV^\ell(C_d)$ containing $L_{1,1}$ and $\mathcal{O}^{(j)}(\gamma_d(0))$, which be definition is $O^\ell(j,1)$.
    The proof for $\Sing(O^\ell_{1,j})$ is analogous after a change of basis, interchanging the roles of $\gamma_d(0)$ and $\gamma_d(1)$.
\end{proof}

\begin{proposition}\label{prop:sing_XL0^l}
    Let $0<i,j<d-1$ and $\ell\geq i+j$. Then, for $\ell\leq d-1$ we have
    \begin{align}
    \Sing(O^\ell(i,j))= O^{\ell+1}(i,j)
    \end{align}
    and for $\ell\geq d$ the singular locus is empty.
\end{proposition}
\begin{proof}
Let $P$ be the plane spanned by the two osculant planes $\mathcal{O}^{(i)}(\gamma_d(0))$ and $\mathcal{O}^{(j)}(\gamma_d(1))$, and let $\mathcal{V}_{P}=\{V\mid P\subset V\}\subset\gr(d{-}1,d{+}1)$.
Notice that the generalized Vandermonde determinant, \eqref{eq:generalized_vandermonde}, implies that $\dim(P)=i+j$.
Therefore, let $\pi:\C^{d+1}\rightarrow\C^{d-i-j+1}$ be the projection away from $P$.  
Then, $\pi$ induces the regular map 
    \[
    \overline{\pi}:\mathcal{V}_{P}\subset\gr(d{-}1,d{+}1)\rightarrow\gr(d{-}i{-}j{-}1,d{-}i{-}j{+}1)
    \]
    which in fact is an isomorphism. We then get the following identification   
    \[
    O^\ell(i,j)\cong\overline{\pi} (\mathcal{V}_{P}\cap \secV^\ell(C_d))
    =\secV^{\ell-i-j}(C_{d-i-j}).
    \]
    Therefore, the singular locus is given by Proposition \ref{prop:it_sing_locus_chow} as $\secV^{\ell-i-j+1}(C_{d-i-j}).$
    This lifts under $\overline{\pi}$ to $O^{\ell+1}(i,j)$.
    The claim for $\ell\geq d$ is trivial as then both $\secV^{\ell}(C_d)$ and $O^\ell(i,j)$ are empty.
\end{proof}

A summary of our results from this section is given in Table \ref{tab:strat_summary}.
\begin{table}[H]
\centering
    \begin{tabular}{c|c|c}
        name & codimension & geometric interpretation\\ \hline \hline
            $\CH(\Sec)$ & $1$ & $k$-spaces meeting $\Sec$  \\ \hline
              $\secV^\ell(C_{k+1})$ & $\ell+1$ &  $k$-spaces meeting $C_{k+1}$ to order $\ell$  \\ \hline
             \multirow{2}{1em}{$O^\ell_{0,j}$} & \multirow{2}{*}{$\ell+j+1$}  & $k$-spaces containing the $j$-th osculant plane at $\gamma_{k+1}(0)$,\\
             & & meeting $L_{0,j}$ and intersecting $C_{k+1}$ to order $\ell$\\ \hline
             \multirow{2}{1em}{$O^\ell_{1,j}$} & \multirow{2}{*}{$\ell+j+1$} &  $k$-spaces containing the $j$-th osculant plane at $\gamma_{k+1}(1)$,\\
             & &  meeting $L_{1,j}$ and intersecting $C_{k+1}$ to order $\ell$ \\ \hline
             \multirow{2}{*}{$O^\ell(i,j)$} & \multirow{2}{*}{$\ell+j+i$} &  $k$-spaces containing the $i$-th osculant plane at $\gamma_{k+1}(0)$,  \\
             & & the $j$-th at $\gamma_{k+1}(1)$ and intersecting $C_{k+1}$ to order $\ell$ \\ \hline
        \end{tabular}
        \caption{Summary of the stratification of the algebraic boundary $\partial_a\aA^\infty_k$ with codimension in the Grassmannian $\gr(k,k{+}2)$.}
        \label{tab:strat_summary}
    \end{table}

\section{Residual Arrangement}\label{sec:residual_arrangement}
The \emph{residual arrangement} $\mathcal{R}(\aA^\infty_k)$ of the limit amplituhedron $\aA^\infty_k$ is defined as the part of the singular locus of its algebraic boundary $\partial_a\aA^\infty_k$ which does not intersect the Euclidean boundary $\partial\aA^\infty_k$. The main result of this section is thus:

\begin{theorem}\label{thm:res_arr_A_infty}
    The residual arrangement $\mathcal{R}(\aA^\infty_k)$ is empty.
\end{theorem}

In order to prove this claim we show that $\aA_k^\infty$ intersects each stratum discovered in Section \ref{sec:stratification} in a dense subset (in the Zariski topology), type by type.

\begin{proposition}\label{prop:strata_denseness}
    For $\ell\leq k$ let $Y$ be a linear $\ell$-space in $\C^{k+2}$ spanned by $\ell$ rows of $Z$, and $\mathcal{V}_Y$ be the subvariety of $\gr(k,k{+}2)$ comprised of all $k$-spaces containing $Y$. Then, $\aA^\infty_k$ intersects $\mathcal{V}_Y$ in a Zariski dense subset.
\end{proposition}
\begin{proof}
    Let $Y$ be spanned by the rows $Z_{i_1},\ldots,Z_{i_\ell}$ of $Z$. Further let, $\pi_Y$ be the projection of $\C^{k+2}$ away from $Y$, and $\pi'_Y$ be the projection of $\C^n$ away from the span of $e_{i_1},\ldots, e_{i_\ell}$, where $e_j$ is the $j$-th standard basis (row) vector of $\C^n$.
    With the notation $\mathcal{V}_Y$ as above, we have $\dim(\mathcal{V}_Y)=2(k-l)$. Since the rows $Z_j$ of $Z$ lie on the rational normal curve $C_{k+1}$, the projection induces the isomorphism $\pi_Y(C_{k+1})\cong C_{k+1-\ell}$ in $\PP^{k+1-\ell}$. 
    By deleting the rows $Z_{i_j}$ of $Z$ we get a totally positive $(n{-}\ell)\times (k{-}\ell{+}2)$  matrix $Z\vert_Y$, inducing a amplituhedron map
    \[
    \widetilde{Z\vert_Y}:\gr(k{-}\ell,n{-}\ell)_{\geq 0}\rightarrow \aA^I_{k-\ell}\subset\gr(k{-}\ell,k{-}\ell{+}2),
    \]
    sending $V'\mapsto V'Z\vert_Y$. Its image is a semialgebraic set with non-empty interior i.e. of dimension $2(k-\ell)$.
    Now, let $V\in\gr(k,n)_{\geq 0}$ such that there is the $\ell\times\ell$ identity matrix in the columns $i_1,\ldots, i_\ell$. We identify the remaining columns with a matrix  $V'\in\gr(k{-}\ell,n{-}\ell)_{\geq 0}$ via the map $\pi_Y'$.
    Then $\widetilde{Z}(V)=VZ$ contains $Y$ and thus $VZ\in\mathcal{V}_Y$.
    More precisely, the image under $\widetilde{Z}$ can be written as the rowspan of the matrix obtained by concatenation of $Y$ and $\widetilde{Z}\vert_Y(V')$.  
    Therefore, we get a semialgebraic set contained in $\aA^\infty_k\cap \mathcal{V}_Y$ with non-empty interior, proving the claim.
\end{proof}

\begin{lemma}\label{lem:SingCH_dense}
        For all $\ell\leq k$ the limit amplituhedron $\aA^\infty_{k}$ intersects the variety of $\ell$-secants $\secV^\ell(C_{k+1})$ in a Zariski dense subset.
\end{lemma}
\begin{proof}
    Pick real numbers $0<s_1 < s_2 < \ldots < s_\ell <1$, and define the $\ell$-space 
    \[
    Y(s) = \operatorname{rowspan}
    \begin{pmatrix}
        \gamma_{k+1}(s_1) \\ \vdots \\ \gamma_{k+1}(s_\ell)
    \end{pmatrix}
    .
    \]
    It follows from Proposition \ref{prop:strata_denseness} that $\aA^\infty_k$ intersects $\mathcal{V}_{Y(s)}$ in a dense subset. Any element of $\mathcal{V}_{Y(s)}$ is an element of $\Sing^\ell(\CH(C_{k+1})=\secV^\ell(C_{k+1})$ by construction. Since the fiber of $\mathcal{V}_{Y(s)}$ over $s$ has dimension $2k{-}l =  \dim(\secV^\ell(C_{k+1}))$, this shows the claim.
\end{proof}

\begin{lemma}\label{lem:CHcapH_dense}
    For all $\ell\leq k$ the limit amplituhedron $\aA^\infty_k$ contains a Zariski dense subset of the intersection $\secV^\ell(C_{k+1})\cap\CH(\Sec)$.
\end{lemma}
\begin{proof}
    Again pick real numbers $0<s_1 < s_2 < \ldots < s_{\ell-1} <1$, and define the $\ell$-space 
    \[
    Y(0,s) = \operatorname{rowspan}
    \begin{pmatrix}
        \gamma_{k+1}(0) \\ \gamma_{k+1}(s_1) \\ \vdots \\ \gamma_{k+1}(s_{\ell-1})
    \end{pmatrix}.
    \]
    It follows from Proposition \ref{prop:strata_denseness} that $\aA^\infty_k$ intersects $\mathcal{V}_{Y(0,s)}$ in a dense subset and any element of $\mathcal{V}_{Y(0,s)}$ is an element of $\secV^\ell(C_{k+1})\cap\CH(\Sec)$ by construction. Since the fiber of $\mathcal{V}_{Y(0,s)}$ over $s$ has dimension $2k{-}l{-}1$, which equals the dimension of $\secV^\ell(C_{k+1})\cap\CH(\Sec)$, this proves the claim.
\end{proof}

\begin{lemma}\label{lem:X_L0_dense}
For all $j<\ell\leq k$ the limit amplituhedron $\aA^\infty_k$ intersects both $O^\ell_{0,j}$ and $O^\ell_{1,j}$ in Zariski dense subsets. 
\end{lemma}
\begin{proof}
    Recall that $\dim(O^\ell_{0,j})=2k-\ell-j-1$ and set $c=\ell-j-1$.
    Fix $s=(s_1,\ldots,s_{c})\in(0,1)^{c}$, take a partition $I$ of $[0,1]$ refining $s$, and let $j<i_1<\ldots<i_c$ be the labels of the $s_i$'s in $I$.
    Consider a family of partitions $\mathcal{I}=\{I_n\}_{n\in\N}$ with $I_n\leq I_{n+1}$ for all $n$ refining $I$, in such a way that $\gamma_{k+1}(0),\ldots, \gamma_{k+1}(t_j)$ span the osculant plane $\mathcal{O}^{(j)}(\gamma_{k+1}(0))$ for $n\rightarrow\infty$.
    Next consider the $\ell$-space
    \begin{align}\label{eq:Y_L0}
    Y(I_n) = \operatorname{rowspan}
    \begin{pmatrix}
        \gamma_{k+1}(0) \\ \vdots \\  \gamma_{k+1}(t_j) \\ \gamma_{k+1}(t_{i_1}) \\ \vdots \\ \gamma_{k+1}(t_{i_{c}}) \\ \gamma_{k+1}(1)
    \end{pmatrix}.
    \end{align}
    Notice that $Y(I_n)$ intersects $L_{0,j}$ in the limit for $n\rightarrow\infty$ by construction.
    As usual, let $\mathcal{V}_{Y(I_n)}$ denote the set of $k$-spaces containing $Y(I_n)$ for any fixed $n\in\N$. Then, notice that $\dim(\mathcal{V}_{Y(I_n)})=2k-2\ell$ and therefore the union of fibers over all $s$ is of dimension
    $2k-l-j-1$ which equals $\dim(O^\ell_{0,j}).$
    Consequently, $\mathcal{V}_{Y(I_n)}$ is dense in $O^\ell_{0,j}$ for $n\rightarrow\infty$, and thus $\aA^\infty_k$ must in turn be dense in $\mathcal{V}_{Y(I_n)}$ for all $n\in\N$ by Proposition \ref{prop:strata_denseness}. 
    It follows that $\aA^\infty_k$ is dense in $O^\ell_{0,j}$, as ought to be shown.
    The proof for $O^\ell_{1,j}$ is analogous upon defining 
    \begin{align}\label{eq:Y_L1}
        Y(I_n) = \operatorname{rowspan} 
        \begin{pmatrix}
            \gamma_{k+1}(0) \\ \gamma_{k+1}(t_{i_1}) \\ \vdots \\ \gamma_{k+1}(t_{i_{c}}) \\ \gamma_{k+1}(t_{n-j}) \\ \vdots \\ \gamma_{k+1}(1)
        \end{pmatrix}.
    \end{align}    
\end{proof}

\begin{lemma}\label{lemma:intersection_strata_dense}
    The limit amplituhedron $\aA^\infty_k$ is Zariski dense in any intersection of the strata from Section \ref{sec:strata}.
\end{lemma}
\begin{proof}
    For $\CH(\Sec)$, $\secV^\ell(C_{k+1})$ and their intersection $\CH(\Sec)\cap\secV^\ell(C_{k+1})$ this follows from Lemma \ref{lem:SingCH_dense} and \ref{lem:CHcapH_dense}.
    For $O^\ell(i,j)$ the subsets defined in \eqref{eq:Y_L0} and \eqref{eq:Y_L1} can be combined to show that the intersection with $\aA^\infty_k$ is Zariski dense as well.
    Similarly, for the intersections $O^\ell(i,j)\cap O^\ell_{0,m}$ and $O^\ell(i,j)\cap O^\ell_{1,m}$.
    Finally, as $O^\ell_{0,i}\cap O^{\ell'}_{1,j}=O^{\operatorname{max}(\ell,\ell')}(i,j)$ by definition, this proves the claim.
\end{proof}

Now we can prove the main result of this section, Theorem \ref{thm:res_arr_A_infty}.

\begin{proof}[Proof of Theorem \ref{thm:res_arr_A_infty}]
    Combining Theorem \ref{thm:algebraic_boundary_of_A_infty} with Lemmas \ref{lem:SingCH_dense}-\ref{lemma:intersection_strata_dense} proves the claim.
\end{proof}

The remainder of this section is devoted to proving Theorem \ref{thm:main_pos_geom}, to that end we want to employ the notion of adjoints, as recent generalization of this notion, see e.g. \cite{polypols,ranestad2024adjoints}, have been a successful tool in the study of positive geometries. For our purposes an adjoint is defined as follows.

\begin{definition}\label{def:adjoint}
    Let $R$ be the coordinate ring of the complex Grassmannian $\gr_\C(k,k{+}2)$, and $R_0$ its homogeneous degree $0$ part.
    Then, $\alpha(\aA^\infty_k)\in R_0$ is an \emph{adjoint} for $\aA^\infty_k$ if  $\alpha(\aA^\infty_k)(Y)=0$ for all $Y\in\mathcal{R}(\aA^\infty_k)$.
\end{definition}
\begin{remark}
    The adjoint can be identified with a rational section of the canonical bundle on $\gr(k,k{+}2)$.
    Therefore, the degree of the adjoint is essentially determined by the sum of the degree of the canonical divisor of $\gr(k,k{+}2)$ and the degree of its algebraic boundary components. 
    In our case, the degree of the canonical divisors is $-k-2$, see \cite[Section 5.7.4.]{and_all_that}, and by Theorem \ref{thm:algebraic_boundary_of_A_infty} the degrees of the components of the algebraic boundary add up to $(k+1)+1=k+2$. 
    Therefore, we get that $\deg(\alpha(\aA^\infty_k))=k+2-k-2=0$. 
\end{remark}

Then, Theorem \ref{thm:res_arr_A_infty} has the following immediate consequence.

\begin{corollary}\label{cor:unique_adjoint}
    The up to multiplication with a constant unique adjoint $\alpha(\aA^\infty_k)$ interpolating the residual arrangement $\mathcal{R}(\aA^\infty_k)$ equals $1$. \qed
\end{corollary}

With the intention of proving the existence of a canonical form for $\aA^\infty_k$ we make the following observations.
\begin{proposition}\label{prop:1D_strata}
    Let $k\geq 3$. There are precisely $k$ boundary strata of $\aA^\infty_k$ of dimension $1$. They are $O_{0,k-1}^{k-1}$, $O_{1,{k-1}}^{k-1}$, and $O^k(i,j)$ for all $1\leq i,j\leq k-2$ with $i+j=k-1$.
\end{proposition}
\begin{proof}
    First notice that $\secV^\ell(C_{k+1})$ is of codimension $k$ at most, and intersecting with $\CH(\Sec)$ only increases the codimension by $1$. Therefore, all higher codimension parts come from intersections with $O_{0,i}^j$ and $O_{1,i}^j$ for some suitable $i$ and $j$.
    Now it is easy to see that both $O_{0,k-1}^{k-1}$ and $O_{1,k-1}^{k-1}$ are two distinct $1$-dimensional strata of $\partial_a\aA^\infty_k$, as they are containing the $(k{-}1)$-dimensional osculant plane at $\gamma_{k+1}(0)$ and $\gamma_{k+1}(1)$, respectively while also need to intersect $C_{k+1}$ in an additional point.
    Fix $1\leq\ell\leq k-1$. Then, the $k$-spaces in the intersection $O^\ell(i,j)=O^k_{0,i}\cap O^k_{1,j}$ are precisely those which contain the two osculant planes $\mathcal{O}^{(i)}(\gamma_{k+1}(0))$ and $\mathcal{O}^{(j)}(\gamma_{k+1}(1))$ and additionally meet $C_{k+1}$ in another point. 
    Notice, that the additional condition of meeting the lines $L_{0,i}$ and $L_{1,j}$ is implied as $\gamma_{k+1}(1)$ is contained in the former and $\gamma_{k+1}(0)$ in the latter.
\end{proof}

\begin{lemma}\label{lem:1D_skeleton_connected}
    The union of the $1$-dimensional strata of the algebraic boundary $\partial_a\aA^\infty_k$ is connected.
\end{lemma}
\begin{proof}
    For $k=1$, the claim is discussed in Example \ref{ex:pizza_slice}. 
    For $k=2$, the claim follows from Table \ref{tab:alg_bd_strata} in Example \ref{ex:twisted_cubic}, as all $1$-dimensional strata contain the point $\Sec$.
    For $k>2$ we argues as follows. Recall the strata from Proposition \ref{prop:1D_strata}
    \[
    O^k(i,j)=O^k_{0,i}\cap O^k_{1,j},
    \]
    for all $ 1\leq i,j\leq k-1$ with $i+j=k-1$.
    Then, simply observe that for each $i$ and $j$, there is a unique $k$-space containing $\mathcal{O}^{(i)}(\gamma_{k+1}(0))$ and $\mathcal{O}^{(j)}(\gamma_{k+1}(1))$, which is contained in the two $1$-dimensional strata $O^k(i{-}1,j)=O_{0,{i-1}}^k\cap O^k_{1,j}$ and $O^k(i,j{-}1)=O^k_{0,i}\cap O_{1,{j-1}}^k$.
    This is because the strata are Zariski closed so we can assume the additional point on $C_{k+1}$ to become the $i$-th i.e. $j$-th derivative of $\gamma_{k+1}$ at either $0$ or $1$, respectively.
    Finally, for $O_{0,k-1}^{k-1}$ notice that by the same reasoning as above, it contains the $k$-space spanned by $\mathcal{O}^{(k-1)}(\gamma_{k+1}(0))$ and $\mathcal{O}^{(1)}(\gamma_{k+1}(1))=\gamma_{k+1}(1)$, which is also contained in the $1$-dimensional stratum $O_{0,{k-2}}^k\cap O_{1,{1}}^k$.
    Analogous reasoning for $O_{1,k-1}^{k-1}$ proves the claim.
\end{proof}

We can now prove the following version of our main result Theorem \ref{thm:main_pos_geom}.

\begin{theorem}\label{thm:A_infty_pos_geom}
    The limit amplituhedron $(\gr(k,k{+2}),\aA^\infty_k)$ is a positive geometry. 
    In particular, its up to scaling unique canonical form $\Omega(\aA^\infty_k)$ can be expressed via the Chow forms of the rational normal curve $C_{k+1}$ and its secant line $\Sec$ as  
    \begin{align}\label{eq:can_form}
        \Omega(\aA^\infty_k)=\frac{1}{\operatorname{CF}(C_{k+1})\operatorname{CF}(\Sec)}dx_1\wedge\dots\wedge dx_{2k},
    \end{align}
    where the $x_i$ are local coordinates of the Grassmannian $\gr(k,k{+}2)$ away from the hyperplane $\CH(\Sec)$.
\end{theorem}
\begin{proof}
    By Corollary \ref{cor:unique_adjoint} the numerator of a canonical form, if it exists, must be a constant.
    Similarly, by Theorem \ref{thm:algebraic_boundary_of_A_infty} the denominator of the canonical form is fixed, by demanding that it has simple poles along the boundary.
    In combination these observations imply that any contender for a canonical form must be non-zero multiple of $\Omega(\aA^\infty_k)$ as in \eqref{eq:can_form}, thus if such a form exists it is unique up to scaling.
    So it remains to check whether there is a way to orient the boundary such that taking successive residues of $\Omega(\aA^\infty_k)$ along faces of $\aA^\infty_k$ results in $\pm1$ at the vertices.
    Choosing an orientation on the algebraic boundary $\partial_a\aA^\infty_k$ induces an orientation on the flag of the strata containing a vertex of the residual arrangement. 
    The canonical form of $\aA^\infty_k$ restricted to a $1$ dimensional stratum must be a non-zero $1$-form with simple poles along the vertices.
    Therefore, we can choose the scalar in the numerator of $\Omega(\aA^\infty_k)$ in a way which makes the residue equal to $+1$ at one vertex and $-1$ at the other vertex of that $1$-dimensional stratum.
    By Lemma \ref{lem:1D_skeleton_connected} we know that the union of all $1$-dimensional strata is connected, thus this gives a compatible way of assigning the orientations everywhere.
    This shows existence and thus uniqueness of the canonical form for $\aA^\infty_k$, and therefore that $(\gr(k,k{+2}),\aA^\infty_k)$ is a positive geometry.
\end{proof}


\printbibliography

\noindent{\bf Authors' addresses:}
\medskip

\noindent Joris Koefler, MPI-MiS Leipzig
\hfill {\tt joris.koefler@mis.mpg.de}

\noindent Rainer Sinn, University of Leipzig
\hfill {\tt rainer.sinn@uni-leipzig.de}
\end{document}